\documentclass[12pt]{amsart}
\usepackage{amsfonts, amsbsy, amsmath, amssymb}

\usepackage[T1]{fontenc}
\usepackage[utf8]{inputenc}
\usepackage{lmodern}
\usepackage{amsmath, amsthm, amssymb,amscd, mathrsfs, amsfonts, mathtools,tikz-cd}
\usepackage{relsize}
\usepackage{euler}
\usepackage{times}
\usepackage[all]{xy}
\usepackage{todonotes}
\usepackage{xcolor}
\usepackage[bookmarks=true,
            bookmarksnumbered=true, breaklinks=true,
            pdfstartview=FitH, hyperfigures=false,
            plainpages=false, naturalnames=true,
            colorlinks=true,pagebackref=true,
            pdfpagelabels]{hyperref}

\usepackage{tikz}
\usepackage{pgfplots}

\pgfplotsset{compat=1.10}
\usepgfplotslibrary{fillbetween}
\usetikzlibrary{patterns}
\usetikzlibrary{matrix}

\usepackage[lite]{amsrefs}

\renewcommand{\PrintDOI}[1]{\href{http://dx.doi.org/\detokenize{#1}}{doi: \detokenize{#1}}%
  \IfEmptyBibField{pages}{, (to appear in print)}{}}

\def\commutatif{\ar@{}[rd]|{\circlearrowleft}}

\newtheorem{thm}{Theorem}[section]
\newtheorem{pro}[thm]{Proposition}

\newtheorem{cor}[thm]{Corollary}

\theoremstyle{definition}
\newtheorem{df}[thm]{Definition}

\theoremstyle{remark}
\newtheorem{rmk}[thm]{Remark}

\newtheorem{ex}[thm]{Example}

\allowdisplaybreaks

\newcommand{\Z}{\mathbb{Z}}

\newcommand\rt{\triangleright}

\usepackage{bbm}

\hypersetup{
  colorlinks = true,
  urlcolor = blue,
  linkcolor = blue,
  citecolor = red,
  pdfauthor = {Elhamdadi, M., Fernando, N.},
  pdfkeywords = {Quandles, ring quandle},
  pdftitle = {Elhamdadi, M., Fernando, N. - Ring Theoretic Aspects of Quandles},
  pdfsubject = { racks, quandles},
  pdfpagemode = UseNone
}

\author{Boris Tsvelikhovskiy} 
\address{Department of Mathematics, 
	Northeastern University, Boston MA 02115} 
\email{tsvelikhovskiy.b@husky.neu.edu}

\title{Nontrivial Topological Quandles}

\begin{document}

\maketitle 
\begin{abstract}
We show that there are infinitely many nonisomorphic quandle structures on any topogical space $X$ of positive dimension. In particular, we disprove Conjecture $5.2$ in \cite{CES}, asserting that there are no nontrivial quandle structures on the closed unit interval $[0,1]$.

 \end{abstract}
\section{Introduction}
Quandles are generally non-associative algebraic structures (the exception being the trivial quandles).  They were introduced independently in the 1980's by Joyce \cite{Joyce} and Matveev \cite{Matveev} with the purpose of constructing invariants of knots in the three space and knotted surfaces in four space.  However, the notion of a quandle can be traced back to the 1940's in the work of Mituhisa Takasaki \cite{Takasaki}.  The three axioms of a quandle algebraically encode the three Reidemeister moves in classical knot theory.  For a recent treatment of quandles (see \cite{EN}).  Joyce and Matveev introduced the notion of the fundamental quandle of a knot and gave a theorem that brings the problem of equivalence of knots to the problem of the quandle isomorphism of their fundamental quandles.  Precisely, two knots $K_1$ and $K_2$ are equivalent (up to reverse and mirror image) if and only if the fundamental quandles $Q(K_1)$ and  $Q(K_2)$ are isomorphic. 
\par
Recently, there has been investigations of quandles from algebraic point of views and their relations to other algebraic structures such as Lie algebras \cite{CCES1,CCES2}, Leibniz algebras \cite{Kinyon,KW}, Frobenius algebras and Yang Baxter equation \cite{CCEKS}, Hopf algebras \cite{AG}, transitive groups \cite{Vend}, quasigroups and Moufang loops \cite{Elhamdadi}, ring theory \cite{BPS,EFT} etc. 

\par
The notion of topological quandles was introduced by Rubinsztein in \cite{Rub}. A \textit{topological rack} $X$ is a topological space with a rack binary operation $f(x,y):
X \times X \rightarrow X$, s.t. $f(x,y)$ is continuous with respect to the
topological structure, the right multiplication $R_x: X\rightarrow X, y \mapsto f(y,x)$ is a
homeomorphism for any $x\in X$, and satisfies the right distributivity  $f(f(x,y),z)=f(f(x,z),f(y,z)) ~ \forall x, y, z \in X$. The initial paper contained plenty of examples of such structures (see examples $2.1-2.8$ therein). Using the braid group action $B_n$ on the cartesian product of $n$ copies of a  topological quandle $(Q,f)$ defined on the generators $\sigma_{ii+1}\in B_n$ via $\sigma_{ii+1}(x_1,\hdots x_n)=(x_1,\hdots, x_i,f(x_i,x_{i+1}) ,x_{i+2}\hdots,x_n)$, the author associates  the space $J_{(Q,f)}(L)$ of fixed points under the action of the braid $\sigma\in B_n$ for the element $\sigma$ corresponding to the oriented link $L$. The main result of the paper was that every topological quandle  the space $J_{(Q,f)}(L)$ for an oriented link $L$  depends only on the isotopy class of the oriented link $L$. (see Sections $3,4$ of \cite{Rub} for details).
\par
The main goal of this paper is to show how to produce topological quandle structures on topological manifolds.
\par
Our exposition is organized as follows. In Section $2$ we recall the definition and basic concepts of quandles with examples. The core of the paper is Section $3$, where, after recalling the generalities on topological quandles, we explain a construction, which  produces nontrivial topological quandle structures on topological manifolds. Applying this construction allows to obtain such structures on the closed interval $[0,1]$ and, using that any two closed intervals are homeomorphic, on any closed interval $[a,b]$. In particular, this implies that Conjecture $5.2$ arisen in \cite{CES} that the only quandle operation on a closed interval is the trivial one was wrong. Furthermore, it is shown that there are infinitely many nonisomorphic topological quandle structures on the closed interval and, in general, on any topological manifold of dimension greater than zero. 
\par
In Section $4$ we make concluding remarks and propose possible directions for further investigation of the subject.

\textbf{Acknowledgement:}
I would like to thank Mohamed Elhamdadi for introducing me to the subject, bringing my attention to Conjecture $5.2$ in \cite{CES} and warm hospitality during my stay at the University of South Florida. I also thank Mohamed Elbehiry for stimulating discussions and plenty of helpful suggestions on the improvement of the exposition.
\section{Review of Quandles}\label{review}
We start this section by giving the basics of quandles with examples.
\begin{df}\label{quandledef}
A {\it quandle}, $X$, is a set with a binary operation $(a, b) \mapsto  a \rt b$ such that

(I) For any $a \in X$,
$a\rt a =a$.

(II) For any $a,b \in X$, there is a unique $c \in X$ such that
$a= c\rt b$.

(III)
For any $a,b,c \in X$, we have
$ (a \rt b) \rt c=(a\rt c)\rt (b\rt c). $
\end{df}

\noindent A {\it rack} is a set with a binary operation that
satisfies (II) and (III). Racks and quandles have been studied
extensively in, for example, \cite{Joyce,Matveev}.  For more details on racks and quandles see the book \cite{EN}.

The following are typical examples of quandles: 
\begin{itemize}
\item
A group $G$ with
conjugation as the quandle operation: $a \rt b = b^{-1} a b$,
denoted by $X=$ Conj$(G)$, is a quandle. 

\item
Any subset of $G$ that is closed under such conjugation is also a quandle. More generally if
$G$ is a group, $H$ is a subgroup, and $\sigma$ is an automorphism that
fixes the elements of $H$  ({\it i.e.} $\sigma(h)=h \ \forall h \in
H$), then $G/H$ is a quandle with $\rt $ defined by $Ha\rt
Hb=H \sigma(ab^{-1})b.$ 

\item
Any ${\Z }[t, t^{-1}]$-module $M$ is
a quandle with $a\rt b=ta+(1-t)b$, for $a,b \in M$, and is called
an {\it  Alexander  quandle}. 

\item
Let $n$ be a positive integer, and
for elements  $i, j \in \Z_n$, define $i\rt j = 2j-i \pmod{n}$. Then $\rt$ defines a quandle structure
called the {\it dihedral quandle}, and denoted by $R_n$, that
coincides with the set  of reflections in the dihedral group
with composition given by conjugation.

\item
Any group $G$ with the quandle operation: $a \rt b = ba^{-1}  b$ is a quandle called Core(G). 
\end{itemize}

The notions of quandle homomorphims and automorphisms are clear.  Let $X$ be a quandle, thus the second axiom of Definition~\ref{quandledef} makes any right multiplication by any element $x$, $R_x: y \mapsto y \rt x$, into a bijection .  The third axiom of Definition~\ref{quandledef} makes $R_x$ into a homomorphism and thus an automorphism.  Let $Aut(X)$ denotes the group of all automorphisms of $X$ and let $Inn(X):=<R_x, \; x \in X>$ denotes the subgroup generated by right multiplications.  The quandle $X$ is called \textit{connected} quandle if the group $Inn(X)$ acts transitively on $X$, that is, there is only one orbit.  
\section{Quandle structures on topological manifolds}
\begin{df}
\label{TopQuandle}
A \textit{topological rack} $X$ is a topological space with a rack binary operation $f(x,y):
X \times X \rightarrow X$, s.t. $f(x,y)$ is continuous with respect to the
topological structure, the right multiplication $R_x: X\rightarrow X, y \mapsto f(y,x)$ is a
homeomorphism for any $x\in X$, and satisfies the right distributivity:
\begin{equation} 
\label{distributivity}
f(f(x,y),z)=f(f(x,z),f(y,z)) ~ \forall x, y, z \in X. 
\end{equation} 
If, in addition, $f(x,x)=x$ for each
$x\in X$, then we say that $X$ is a \textit{topological quandle} (see examples $2.1-2.8$ in \cite{Rub}).
\end{df}

Next we will present a nontrivial  topological quandle structure on the unit interval $[0,1]$. Let us explain one method how to construct such structures in general. Consider a topological space $X=X_1 \cup X_2$ with the trivial quandle operation $f(x_2,x_1)=x_2$ if both points $x_1, x_2$ are in $X_1$ or in $X_2$ or $x_1 \in X_1$ and $x_2 \in X_2$. Suppose there exists a map $\varphi: X_2 \rightarrow Homeo(X_1)$, s.t. the image of $\varphi$ is a nontrivial commutative subgroup under composition. In addition assume that $\varphi (x_2) = id$, whenever $x_2\in X_1 \cap X_2$ and the function $R_{x_2}:X\rightarrow X$ given by 
\begin{equation}
R_{x_2}(x)=\begin{cases} x, \mbox{ for } x \in X_2 \\  \varphi_{x_2}(x) \mbox{ for } x \in X_1  \end{cases} 
  \end{equation}
is continuous. Then the function $\psi: X\times X\rightarrow X$ given by 
\begin{equation}
\psi(x,y):=\begin{cases} x, \mbox{ for both } x,y \in X_1  \mbox{ or }  x,y  \in X_2  \mbox{ or }  y \in X_1 \mbox{ and } x \in X_2\\ \varphi_{y}(x)  \mbox{ for } x \in X_1  \mbox{ and } y\in X_2,  \end{cases} 
  \end{equation}
  provided it is continuous, produces a nontrivial topological quandle structure on $X$. Theorem \ref{Counterrr} below is an instance of this construction for $X=[0,1]$ with $X_1=[0,\frac{1}{2}]$ and $X_2=[\frac{1}{2},1]$. Verification that the function produced satisfies the axioms of Definition \ref{TopQuandle} for general $X$ is completely analogous.
\begin{rmk}
Any quandle $X$ can be made into a topological quandle using the discrete topology on $X$.                                         
\end{rmk}
\begin{thm}
\label{Counterrr}
 The function $f:[0,1]\times [0,1]\rightarrow [0,1]$ given by 
\begin{equation}
f(x,y):=\begin{cases} x, \mbox{ for } y\leq \frac{1}{2} \mbox{ and } x \in [0,1] \mbox{ or } x \geq \frac{1}{2}  \mbox{ and } y \in [0,1] \\  \frac{1}{2}(2x)^{1+\varepsilon}  \mbox{ for } y = \frac{1}{2}+\varepsilon \mbox{ and } x\leq \frac{1}{2}   \end{cases} 
  \end{equation}
  provides a topological quandle structure on the unit interval $[0,1]$.
\end{thm}
\begin{proof}
The properties that $f(x,y)$  is continuous and the right multiplication by $x$ is a homeomorphism for any $x\in  [0,1]$ are straightforward consequences of the definition of $f(x,y)$. Since $\forall x \in [0,1]$ one of the requirements $x\geq \frac{1}{2}$ or $x\leq \frac{1}{2}$ is satisfied, we always have $f(x,x)=x$. Hence, it remains to check distributivity  \eqref{distributivity}.
Notice that $z\leq\frac{1}{2}$ implies $f(f(x,y),z)=f(x,y)$ with $f(x,z)=x$ and $f(y,z)=y$ providing $f(f(x,z),f(y,z))=f(x,y)$, thus, confirming the equality \eqref{distributivity}. For $z>\frac{1}{2}$ it is convenient to do the case by case verification. 
 \begin{enumerate}
 \item[\textbf{Case 1}:] both $x,y \geq \frac{1}{2} $. Then $f(f(x,y),z)=f(x,z)=x$ and $f(f(x,z),f(y,z))=f(x,y)=x$.
 \item[\textbf{Case 2}:] both $x,y \leq \frac{1}{2}$. Then, using that $y \leq \frac{1}{2}$, write $f(f(x,y),z)=f(x,z)$, while $f(f(x,z),f(y,z))=f(x,z)$, as $f(y,z)\leq y  \leq \frac{1}{2}$.
 \item[\textbf{Case 3}:] $x \geq \frac{1}{2}$ and $y \leq \frac{1}{2}$. Then $f(f(x,y),z)=f(x,z)=x$, while $f(f(x,z),f(y,z))=f(x,f(y,z))=x$, where all equalities hold since $x \geq \frac{1}{2}$.
 \item[\textbf{Case 4}:] $x \leq \frac{1}{2}$ and $y \geq \frac{1}{2}$. Using that $y \geq \frac{1}{2}$, we obtain the equality $f(f(x,z),f(y,z))=f(f(x,z),y)$. This allows to rewrite \eqref{distributivity} as 
 \begin{equation} 
\label{distributivityCase4}
f(f(x,y),z)=f(f(x,z),y) ~ \forall x, y, z \in X. 
\end{equation} 
We write $y=\frac{1}{2}+\varepsilon'$ for some $0<\varepsilon'\leq \frac{1}{2}$. Similarly, as $z>\frac{1}{2}$, we will write $z=\frac{1}{2}+\varepsilon$ for some $0<\varepsilon \leq \frac{1}{2}$. We see that both sides of \eqref{distributivityCase4} are equal to $ \frac{1}{2}(2x)^{(1+\varepsilon)(1+\varepsilon')}$.

\end{enumerate}
\end{proof}
\begin{figure}
\begin{center}
\begin{tikzpicture}
\begin{axis}[axis lines=middle,
            xlabel=$x$,
            ylabel=$R_y(x)$,
            enlargelimits,
            ytick=\empty,
            xtick={0,2,4},
            xticklabels={0,$\frac{1}{2}$,1}]
\addplot[name path=F,domain={0:4}] {x} ;

\addplot[name path=G,red,domain={0:2}] {.5*x^2};
\addplot[name path=S,red,domain={0:2}] {2*(.5*x)^1.5};
\addplot[name path=L,red,domain={0:2}] {.25*x^3};

\end{axis}
\end{tikzpicture}
\end{center}
\caption{Right multiplication $R_y$ for the function $f(x,y)$.}
\end{figure}
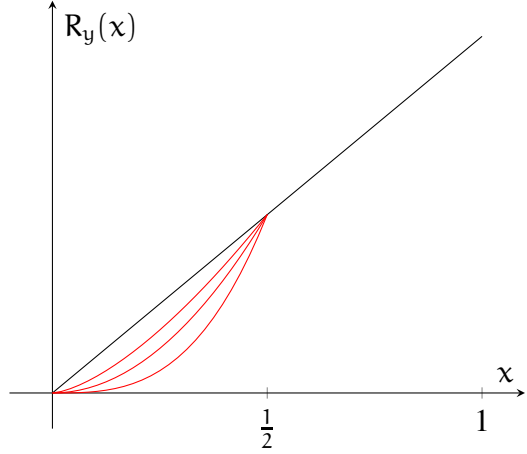
\begin{rmk}
\label{Preserves01}
Instead of $1+\varepsilon$ one can use any continuous function $h(\varepsilon)$ with the only requirement that $x^{h(\varepsilon)}$ is a homeomorphism on the closed interval $[0,\frac{1}{2}]$ for any $\varepsilon \in [0,\frac{1}{2}]$. Recall that every element of the group $Homeo([0,1])$ either preserves the endpoints $0$ and $1$ or switches them. The subgroup of homeomorphisms that preserve the endpoints is denoted by $Homeo^+([0,1])$. It is easy to show that, in fact,  $R_x \in Homeo^{+}([0,1])$ for every $x \in [0,1]$ (see Section $5.1$ in \cite{CES}). This implies that for any quandle structure on the interval $[0,1]$ the endpoints $0$ and $1$ are orbits.
\end{rmk}
\begin{rmk}
\label{UnitBall}
The function $f_{[a,b]}:[a,b]\times [a,b]\rightarrow [a,b]$ given by 
\begin{equation}
f_{[a,b]}(x,y):=\begin{cases} x, \mbox{ for } a\leq y\leq a+\frac{b-a}{2} \mbox{ and } x \in [a,b] \mbox{ or } \frac{(a+b)}{2} \leq x\leq b \mbox{ and } y \in [a,b] \\ a+\frac{b-a}{2}(\frac{2}{b-a}(x-a))^{1+\varepsilon}  \mbox{ for } y =  \frac{a+b}{2} +\varepsilon \mbox{ and } a\leq x\leq \frac{(a+b)}{2}   \end{cases} 
  \end{equation}
  provides a topological quandle structure on the closed interval $[a,b]$. We will also need a slight modification of this function for an open interval $(-1,1)$ given by 
  \begin{equation}
g_{(-1,1)}(x,y):=\begin{cases} x, \mbox{ for } -1< y\leq 0 \mbox{ and } x \in (-1,1) \mbox{ or }0 \leq x< 1 \mbox{ and } y \in (-1,1) \\ -1+(x+1)^{1+y(1-y^2)}  \mbox{ for } y > 0, x \in (-1,0), \\   \end{cases} 
  \end{equation}
  and, more generally, for $B^{\circ}$ an open unit ball in $\mathbb{R}^n$
  \begin{equation}
\Omega_{B^{\circ}}(x,y):=\begin{cases} x, \mbox{ for } -1< y_1\leq 0 \mbox{ and } x_1 \in (-1,1) \mbox{ or }0 \leq x_1< 1 \mbox{ and } y_1 \in (-1,1) \\ (-1+(x_1+1)^{1+y_1(1-\sum\limits_{i=1}^n y_i^2)(1-\sum\limits_{i=1}^n x_i^2)},x_2,\hdots, x_n)  \mbox{ for } y_1 > 0, x_1 \in (-1,0) \\   \end{cases} 
  \end{equation}
 yields a topological quandle structure with $f(x,y) \rightarrow x$ as $x$ or $y$ approach the boundary of the unit ball $\partial B^{\circ}$. Furthermore, 
choosing a homeomorphism $B^{\circ} \rightarrow B_{p,r}^{\circ}$
 yields a topological quandle structure on an open ball $B_{p,r}^{\circ}$ of radius $r$ centered at $p$ with $f(x,y) \rightarrow x$ as $x$ or $y$ approach the boundary $\partial B_{p,r}^{\circ}$, which we denote by $\Omega_{B_{p,r}^{\circ}}(x,y)$.
\end{rmk}

\begin{thm}
\label{InfMany}
There are infinitely many nontrivial nonisomorphic topological quandle structures on a closed interval $[a,b]$ and on the open unit ball $B^{\circ}$.
\end{thm}
\begin{proof}
For any $n\in \mathbb{N}$ consider the points $x_k:=a+\frac{k(b-a)}{n}$ for $k\in \{0,1,\hdots, n-1\}$ and enhance each of the intervals $[x_{k},x_{k+1}]$ with quandle operation provided by the function $f_{[x_{k},x_{k+1}]}$. Complete this to a quandle operation on $[a,b]$ via 
\begin{equation}
F_n(x,y):=\begin{cases} f_{[x_{k},x_{k+1}]}(x,y), \mbox{ if } x,y \in [x_{k},x_{k+1}] \mbox{ for some } k\in \{1,\hdots, n-1\} \\ x, \mbox{ otherwise.} 
 \end{cases}   \end{equation}
Notice, that the set of points $X^n_{triv}:=\{x \in [a,b] | F_n(x,y)=x~ \forall y \in [a,b]\}$ consists of the points $\frac{x_k+x_{k+1}}{2}\leq x\leq x_{k+1}$, i.e. is a disjoint union of $n$ closed intervals. It is clear that for a quandle isomorphism $\varphi: ([a,b], F_s) \rightarrow ([a,b], F_k)$, one must have   $\varphi(X^s_{triv}) = X^k_{triv}$, which is impossible if $k\neq s$, since the map $\varphi$ is a homeomorphism.
\par
Next we verify the assertion for the open unit ball $B^{\circ}$. For any $n\in \mathbb{N}$ consider the points $p_k:=(-1+\frac{2k}{n},0,\hdots,0)$ for $k\in \{1,\hdots, n-1\}$ and enhance each of the open balls $B_{p_k,\frac{1}{n}}^{\circ}$  with quandle operation provided by the function $\Omega_{B_{p_k,\frac{1}{n}}^{\circ}}(x,y)$. Complete this to $\Omega_n(x,y)$, a quandle  operation on $B^{\circ}$ as above for $[a,b]$. The set of points $B^n_{nontriv}:=\{x \in B^{\circ} | F_n(x,y)\neq x~ \mbox{ for some }y \in B^{\circ}\}$ is a disjoint union of $n-1$ open halves of $B_{p_k,\frac{1}{n}}^{\circ}$'s  bounded by western hemispheres. This number is invariant under quandle isomorphisms in the family $(B^{\circ}, \Omega_n(x,y))$, hence, all such quandles are pairwise nonisomorphic.
\end{proof}
\begin{cor}
There are infinitely many nontrivial nonisomorphic topological quandle structures on any topological manifold $X$ of positive dimension.
\end{cor}
\begin{proof}
 Let $X$ be a topological manifold of dimension $n>0$ and $U$ an open subset homeomorphic to the open unit ball $B^{\circ}$ in $\mathbb{R}^n$. The homeomorphism will be denoted by $\varphi$. The function
 \begin{equation}
G(x,y):=\begin{cases} \varphi^{-1}(\Omega_{B^{\circ}}(\varphi(x),\varphi(y))), \mbox{ for } x,y \in U \\ x, \mbox{ otherwise,}  
 \end{cases} 
 \end{equation}
 where $\Omega_{B^{\circ}}$ is the function from Remark \ref{UnitBall}, endows $X$ with a topological quandle structure. Using the construction from Theorem \ref{InfMany}, we provide infinitely many nonisomorphic topological quandle structures on  $X$.
\end{proof}

\begin{ex}
\label{RealLine}
We illustrate how to apply the preceding construction to produce a topological quandle structure on the real line $\mathbb{R}$. Consider the homeomorphism $\varphi: \mathbb{R} \rightarrow (-\frac{\pi}{2},\frac{\pi}{2})$ given by the inverse tangent function $\varphi(x)=arctan(x)$. Then the function
 \begin{equation}
G(x,y):=\begin{cases} tan\left(-\frac{\pi}{2}+\frac{\pi}{2}\left(\frac{2}{\pi}\left(\alpha+\frac{\pi}{2}\right)^{1+\beta}\right)\right) \mbox{ for } x=tan(\alpha)<0, y=tan(\beta), \beta \in [0, \frac{\pi}{4}]  \\  tan\left(-\frac{\pi}{2}+\frac{\pi}{2}\left(\frac{2}{\pi}\left(\alpha+\frac{\pi}{2}\right)^{1+\frac{\pi}{2}-\beta}\right)\right) \mbox{ for } x=tan(\alpha)<0, y=tan(\beta), \beta \in [\frac{\pi}{4},\frac{\pi}{2})  \\ x, \mbox{ otherwise}  
 \end{cases} 
 \end{equation}
 gives a nontrivial topological quandle structure on the real line. 
\end{ex}
The following definition appeared in \cite{CES}.
\begin{df}
The group of inner automorphisms of a topological quandle $Inn(X)$ is the subgroup of $Homeo(X)$ generated by the elements $R_x$ for $x \in X$. If the group $Inn(X)$ acts transitively on $X$ then we say that $X$ is \textit{indecomposable}.
\end{df}
\begin{rmk}
In Section $3$ of \cite{CES} the authors considered affine quandles on $\mathbb{R}$.  These is a family of quandles $(\mathbb{R}, f_t)$ with the quandle structures given by the functions $f_t(x,y) = tx+(1-t)y$ for a fixed $0 \neq t \in \mathbb{R}$. Notice, that the real line with topological quandle structure from Example \ref{RealLine} contains a trivial subquandle $(-\infty,0)$. Therefore, it is not isomorphic to any of the quandles $(\mathbb{R}, f_t)$ (there are no trivial subquandles in $(\mathbb{R}, f_t)$ ).
\par
On the other hand any quandle $(\mathbb{R}, f_t)$ with $t\not\in\{0,1\}$ induces a nontrivial indecomposable topological quandle structure on an open interval $(a,b)$. 
\end{rmk}
\par
The following proposition appeared as Lemma $5.3$  in \cite{CES}. Here we present an alternative proof.
\begin{pro}
There are no nontrivial quandle structures on the closed unit interval $[0,1]$ with $f(x,y)$ a polynomial.
\end{pro}
\begin{proof}
Assume that $f(x,y)=\sum\limits_{i,j}a_{ij}x^iy^j$ is a polynomial function, satisfying the requirements of Definition \ref{TopQuandle}. Then $f(0,0)=0$ implies the constant term $a_{00}=0$. Recall that $f(0,y)=0 \mbox{ and } f(1,y)=1~\forall y \in[0,1]$  (see Remark \ref{Preserves01}). The first equality $f(0,y)=\sum\limits_{j}a_{0j}y^j=0$ implies that the polynomial $f(0,y)$ is identically $0$ and the coefficients $a_{0j}$ vanish, while $f(1,y)=\sum\limits_{i}a_{i0}+yP(y)=1$, where $P(y)$ is a polynomial, implies $yP(y)=0$ and, hence, $f(x,y)$ is a polynomial in $x$ only. The fact that $f(x,x)=x$ concludes the proof.
\end{proof}
\begin{rmk}
Arguing similarly, one can show that the Taylor series of an analytic function  $f(x,y)$ around $(0,0)$ converging on  $[0,1]\times[0,1]$ and providing a quandle structure on the closed unit interval $[0,1]$ must be $f(x,y)=x$.
\end{rmk}
\begin{rmk}
The argument of the  proposition also shows that all rack structures on the closed interval $[0,1]$ with $f(x,y)$ a polynomial function are given by polynomial functions in a single variable $f(x,y)=f(x)$ invertible as functions $[0,1]\rightarrow [0,1]$. 
\end{rmk}
\section{Concluding remarks}
The results of the preceding section imply that the study of rack/quandle structures on topological manifolds should be limited to certain classes of functions. One reasonable restriction might be to consider polynomial functions.

\end{document}